\documentclass{amsart} 

\usepackage{palatino}
\usepackage{graphicx}
\usepackage[T1]{fontenc}
\usepackage[latin9]{inputenc}
\usepackage[table]{xcolor}

\usepackage{amsmath}
\usepackage{mathrsfs}
\usepackage{amsthm}
\usepackage{hyperref}
\usepackage{amsfonts}
\usepackage{lscape}
\usepackage{caption}
\usepackage{amssymb, amsmath} 
\usepackage{graphicx}
\usepackage{tikz}
\usepackage{color}
\usepackage{color,soul}
\usepackage{calc}
\usepackage[normalem]{ulem}
\usepackage{enumitem}
\usepackage[toc,page]{appendix}
\usetikzlibrary{arrows,positioning,automata,shapes,calc,3d,graphs}
\usetikzlibrary{decorations.markings}
\usetikzlibrary{arrows, shadows}
\usepackage[underline=true,rounded corners=false]{pgf-umlsd}
\usepackage{xspace}
\usepackage{verbatim}
\usepackage[margin=1in]{geometry}
\usepackage{xassoccnt}
\usepackage[english]{babel}
\usepackage{pgf,pgfarrows,pgfnodes,pgfautomata,pgfheaps}
\usepackage{amsmath,amssymb}
\usepackage{xcolor}
\usepackage{graphicx}
\usepackage{caption}
\usepackage{subcaption}
\usepackage{float}
\usepackage{wrapfig}
\usepackage{graphicx, scalefnt}
 \usepackage{pgfplots}
 \usepackage{etex}
\usepackage{tikz}

\newtheorem{theorem}{Theorem} 
\newtheorem{lemma}[theorem]{Lemma}

\newtheorem{corollary}[theorem]{Corollary}

\theoremstyle{example}
\newtheorem{example}{Example}

\theoremstyle{definition}
\newtheorem{definition}[theorem]{Definition}
\theoremstyle{remark}

\theoremstyle{plain}
\newtheorem*{repp@theorem}{\repp@title (reformulated)}
\newcommand{\newrepptheorem}[2]{%
\newenvironment{repp#1}[1]{%
 \def\repp@title{#2 \ref{##1}}%
 \begin{repp@theorem}}%
 {\end{repp@theorem}}}
\makeatother
\newrepptheorem{theorem}{Theorem}

\theoremstyle{definition}
\makeatletter
\newtheorem*{rep@theorem}{\rep@title \ continued}
\newcommand{\newreptheorem}[2]{%
\newenvironment{rep#1}[1]{%
 \def\rep@title{#2 \ref{##1}}%
 \begin{rep@theorem}}%
 {\end{rep@theorem}}}
\makeatother
\newreptheorem{example}{Example}

\newcommand{\stkout}[1]{\ifmmode\text{\sout{\ensuremath{#1}}}\else\sout{#1}\fi}

\tikzstyle{vertex}=[circle, draw, inner sep=0pt, minimum size=10pt]
\usetikzlibrary{positioning, fit, calc}   
\tikzset{block/.style={draw, thick, text width=2cm , minimum height=1.3cm, align=center},   
line/.style={-latex}     
}  
\tikzset{block1/.style={text width=2cm , minimum height=1.3cm, align=center},   
line/.style={-latex}     
}  

 \usetikzlibrary{decorations.pathmorphing}
\usetikzlibrary{shapes,arrows}
\usetikzlibrary{calc}
\usetikzlibrary{intersections}
\usetikzlibrary{positioning}
\tikzstyle{decision} = [diamond, draw, text width=5em, text badly centered, node distance=3cm, inner sep=0pt]
\tikzstyle{block2} = [rectangle, draw, text width=5em, text centered, rounded corners, minimum height=4em]

\newdimen\nodeSize
\newdimen\nodeDist
\tikzset{
    position/.style args={#1:#2 from #3}{
        at=(#3.#1), anchor=#1+180, shift=(#1:#2)
    }
}

\usetikzlibrary{positioning,fit,calc} 
\tikzset{block/.style={draw, thick, text width=3cm, minimum height=1.0cm, align=center},   
line/.style={-latex}   
}  
\tikzset{block9/.style={draw, thick, text width=1.5cm, minimum height=1.0cm, align=center},   
line/.style={-latex}   
}  
\tikzset{block10/.style={draw, thick, color=white, text width=3cm, minimum height=1.0cm, align=center},   
line/.style={-latex}   
}

\definecolor{lightgrey}{RGB}{216,227,225}
\usetikzlibrary{positioning,shadows,arrows}
\usepackage{pgfplotstable}
\pgfplotsset{compat=newest}
\definecolor{bblue}{HTML}{4F81BD}
\definecolor{rred}{HTML}{C0504D}
\definecolor{ggreen}{HTML}{9BBB59}
\definecolor{ppurple}{HTML}{9F4C7C}
\definecolor{ppink}{HTML}{FFCC00}

\pgfdeclaremask{computer}{beamer-computer-mask}
\pgfdeclareimage[interpolate=true,mask=computer,height=2cm]{computerimage}{beamer-computer}
\pgfdeclareimage[interpolate=true,mask=computer,height=2cm]{computerworkingimage}{beamer-computerred}

\colorlet{darkred}{red!80!black}
\colorlet{darkblue}{blue!80!black}
\colorlet{darkgreen}{green!80!black}

\def\softness{0.4}
\definecolor{softred}{rgb}{1,\softness,\softness}
\definecolor{softgreen}{rgb}{\softness,1,\softness}
\definecolor{softblue}{rgb}{\softness,\softness,1}

\definecolor{softrg}{rgb}{1,1,\softness}
\definecolor{softrb}{rgb}{1,\softness,1}
\definecolor{softgb}{rgb}{\softness,1,1}

\numberwithin{equation}{section}

\begin{document}

\title{Benford's Law in the ring $\mathbb{Z}(\sqrt{D})$} 
\author{Christine Patterson}
\address{Boise State University, Boise, Idaho 83725}
\email{christinepatterson@boisestate.edu}

\author{Marion Scheepers}
\address{Department of Mathematics, Boise State University, Boise, Idaho 83725}
\email{mscheepe@boisestate.edu}


\date{February 14, 2024} 




\begin{abstract}
For $D$ a natural number that is not a perfect square and for $k$ a non-zero integer, consider the subset  $\mathbb{Z}_k(\sqrt{D})$ of the quadratic integer ring $\mathbb{Z}(\sqrt{D})$ consisting of elements $x+y\sqrt{D}$  for which $x^2 - Dy^2 = k$ . For each $k$ such that  the set $\mathbb{Z}_k(\sqrt{D})$ is nonempty, $\mathbb{Z}_k(\sqrt{D})$ has a natural arrangement into a sequence for which the corresponding sequence of integers $x$, as well as the corresponding sequence of integers  $y$, are strong Benford sequences. 
\end{abstract}

\maketitle


\section{Introduction}

A phenomenon nowadays called Benford's Law was serendipitously discovered independently, years apart, by Simon Newcomb \cite{Newcomb} and Frank Benford \cite{Benford} via observations about the different levels of wear of pages in logarithm books. Two among several iconic tables and graphs are given in Figure \ref{fig:NewcombBenfordTable} and Figure \ref{fig:BenfordDigit1}. 

\begin{center}
\begin{figure}[ht]
\begin{tabular}{|c||c|c|} \hline
Digit & 1st Digit & 2nd Digit\\ \hline  
0      &               &     0.1197 \\ \hline 
1      &   0.3010  &    0.1139 \\ \hline 
2      &   0.1761  &    0.1088 \\ \hline 
3      &   0.1249  &    0.1043 \\ \hline 
4      &   0.0969  &    0.1003 \\ \hline 
5      &   0.0792  &    0.0967 \\ \hline 
6      &   0.0669  &    0.0934 \\ \hline 
7      &   0.0580  &    0.0904 \\ \hline 
8      &   0.0512  &    0.0876\\ \hline  
9      &   0.0458  &    0.0850 \\ \hline
\end{tabular}
\caption{Newcomb's Table, Benford's Table IV}\label{fig:NewcombBenfordTable}
\end{figure}
\end{center}
\vspace{-0.2in}

For a sequence of numbers satisfying Benford's Law, the table in Figure \ref{fig:NewcombBenfordTable} gives the 
frequencies of the most significant nonzero digit, and of the second most significant digit. Figure \ref{fig:BenfordDigit1} is a visual representation, for a sequence satisfying Benford's Law, of the frequencies of the most significant nonzero digit. 

\begin{figure} [ht]
\begin{tikzpicture}[scale=0.8]
    \begin{axis}[
        width  = 1.00*\textwidth,
        height = 8cm,
        major x tick style = transparent,
        ybar=1.5*\pgflinewidth,
        bar width=6pt,
        ymajorgrids = true,
        ylabel = {Percentage Occurrence},
        symbolic x coords={1,2,3,4,5,6,7,8,9},
        xtick = data,
        scaled y ticks = false,
        enlarge x limits=0.15,
        ymin=0,
        legend cell align=left,
        legend columns=9,
        legend style={
                at={(0.5,1.05)},
                anchor=south east,
                column sep=1ex
        }
    ]
        \addplot[style={rred,fill=bblue,mark=none}] 
            coordinates {(1, .3010) (2,0.1761) (3,0.1249) (4,0.0969) (5,0.0792) (6,0.0669) (7,0.0580) (8,0.0512) (9,0.0458)};
    \end{axis}
\end{tikzpicture}
\caption{Probability of occurrence of first digit}\label{fig:BenfordDigit1}
\end{figure}
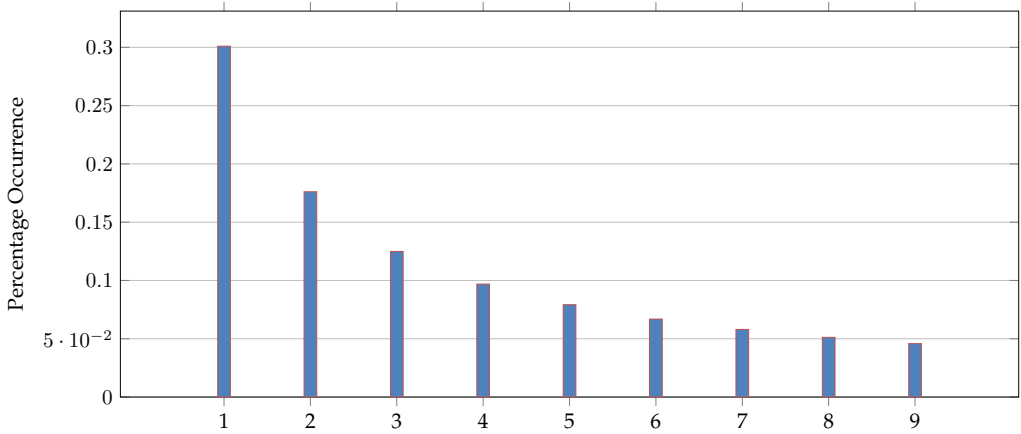

The textbook \cite{BergerHill} gives a broad overview of the pervasive presence of Benford's Law in numerous contexts. In the years since the introduction of Benford's Law it has been proven that certain sequences (for example the Fibonacci sequence and the Lucas sequence - see \cite{LWashington}) satisfy Benford's Law, while others (for example the sequence of prime numbers) do not.  More recently it was shown that for a nonsquare natural number $D$, the sequence of numerators, as well as the sequence of denominators of the sequence of partial fractions of $\sqrt{D}$ each satisfies Benford's Law - see for example \cite{JL}.  Incidentally, these results on partial fractions can also be derived using techniques of our paper. This application will be pursued elsewhere.

In Section 2 we describe the structural features in $\mathbb{Z}(\sqrt{D})$ that will provide the context for the Benford Law analysis. In Section 3 we provide a deeper description of the Benford Law phenomenon and introduce the relevant terminology used in the literature. In Section 4 we introduce a construction that preserves the Benford phenomenon when applied to sequences exhibiting the Benford phenomenon, and use it towards proving presence of Benford's Law in quadratic integer rings.

Regarding the bibliography of this paper: Not all the items are referenced in the body of this paper. These are provided as background information for the reader interested in a more extensive pursuit of the topic of this paper. 

\section{Canonical orbit sequences for $\mathbb{Z}(\sqrt{D})$}\label{sec:CanonicalOrbitSeq}  

For basic terminology about quadratic number fields $\mathbb{Q}(\sqrt{D})$ and their rings of integers $\mathbb{Z}(\sqrt{D})$ we follow \cite{LeVeque}. 
For a positive integer $D$ that is not the square of another integer, consider the quadratic integer ring $\mathbb{Z}(\sqrt{D})$. Each element $\alpha$ of $\mathbb{Z}(\sqrt{D})$ is of the form $\alpha = a + b\sqrt{D}$ where $a$ and $b$ are integers. The corresponding quadratic integer $a - b\sqrt{D}$, called the \emph{conjugate} of $\alpha$, is denoted $\overline{\alpha}$.  
The symbol $\mathcal{N}(\alpha)$ denotes the integer value $\alpha\cdot\overline{\alpha} = a^2 - Db^2$, and is said to be the \emph{norm} of $\alpha$. The set $\mathbb{Z}_k(\sqrt{D}) = \{\alpha\in\mathbb{Z}(\sqrt{D}): \mathcal{N}(\alpha) = k\}$ is said to be a  \emph{norm class} of $\mathbb{Z}(\sqrt{D})$. 
 For a nonzero integer $k$, determining an element $\alpha$  for which $\mathcal{N}(\alpha) = k$ amounts to solving the well-known equation
\begin{equation}\label{eq:DefNorm}
 a^2 - Db^2 = k.
\end{equation}
In recognition of the fundamental contributions of Brahmagupta to the theory underlying Equation (\ref{eq:DefNorm}) we shall refer to this equation as a Brahmagupta equation. 

The norm class $\mathbb{Z}_1(\sqrt{D})$ endowed with the multiplication operation is a group. The group $(\mathbb{Z}_1(\sqrt{D}),\cdot)$ is said to be the \emph{group of units} of $\mathbb{Z}(\sqrt{D})$. For each element $\alpha$ of $\mathbb{Z}_1(\sqrt{D})$, its conjugate $\overline{\alpha}$ is in $\mathbb{Z}_1(\sqrt{D})$ and is the multiplicative inverse of $\alpha$. It is well-known - see for example Corollary 31.5 of \cite{EDB} - 
that the group $(\mathbb{Z}_1(\sqrt{D}),\cdot)$ is isomorphic to the group $\mathbb{Z}_2\times\mathbb{Z}$, and has two generators, namely the idempotent element $-1+0\sqrt{D}$, and an element $x_1+y_1\sqrt{D}\in\mathbb{Z}_1(\sqrt{D})$ with $x_1,\; y_1>0$ for which the pair $(x_1,\;y_1) \in\mathbb{Z}^2$ is lexicographically minimal.  
The element $x_1+y_1\sqrt{D}$ is said to be the \emph{fundamental element} of $\mathbb{Z}_1(\sqrt{D})$. 

More generally, for each nonzero integer $k$ for which the norm class $\mathbb{Z}_k(\sqrt{D})$ is nonempty, consider the elements of the form $x+y\sqrt{D}$ where $x$ and $y$ are non-zero. Among these elements with $y>0$ there is one for which $(x,y)$ is minimum in the lexicographic ordering on $\mathbb{Z}\times\mathbb{Z}$. We define:
\begin{definition}\label{def:fundamentalelt}
For nonzero integer $k$ for which $\mathbb{Z}_k(\sqrt{D})$ is nonempty, let $(x_k,y_k)$ denote the lexicographically minimum element among the $x+y\sqrt{D}$ with $0\le x,\; y$ in the set $\mathbb{Z}_k(\sqrt{D})$. Then $x_k+ y_k\sqrt{D}$ is said to be \emph{the fundamental element of} $\mathbb{Z}_k(\sqrt{D})$.
\end{definition}

For each nonzero integer $k$ for which the norm class $\mathbb{Z}_k(\sqrt{D})$ is nonempty, the elements of the class can be naturally organized into finitely many infinite sequences. Each one of these sequences is of the form $(a_n+b_n\sqrt{D}:n\in\mathbb{N})$ where $a_n$ and $b_n$ are integers. We shall show in Theorem \ref{thm:Benford2} that each of the sequences $(a_n:n\in\mathbb{N})$ and $(b_n:n\in\mathbb{N})$ satisfies the strong Benford law (using the terminology of Diaconis in \cite{Diaconis}). 
 Relevant details are given in Section \ref{sec:BenfordCanonicalOrbits} .
 
Combined with another result (Theorem \ref{thm:conjunction} in Section \ref{sec:Interleaving}) about preservation of the strong Benford property under certain constructions involving several sequences, this result implies that each norm class $\mathbb{Z}_k(\sqrt{D})$ satisfies the strong Benford law (Theorem \ref{thm:Benford2}).

The group $(\mathbb{Z}_1(\sqrt{D}),\cdot)$ of units of $\mathbb{Z}(\sqrt{D})$ acts on the set $\mathbb{Z}(\sqrt{D})$ via multiplication: For $u+v\sqrt{D}$ an element of $\mathbb{Z}_1(\sqrt{D})$, and $x+y\sqrt{D}$ an element of $\mathbb{Z}(\sqrt{D})$, the output of the action of the element $u+v\sqrt{D}$ on $x+y\sqrt{D}$ is the element $(u+v\sqrt{D})\cdot (x+y\sqrt{D})$ of $\mathbb{Z}(\sqrt{D})$. The orbit of the element $x+y\sqrt{D}$ under the action of the group $(\mathbb{Z}_1(\sqrt{D}),\cdot)$ is the set
\begin{equation}\label{eq:orbit}
   Orbit(x+y\sqrt{D}) = \{ (u+v\sqrt{D})\cdot (x+y\sqrt{D}): (u+v\sqrt{D}) \in \mathbb{Z}_1(\sqrt{D})\}.
\end{equation}

The set $\mathbb{Z}(\sqrt{D})\setminus\{0\}$ is partitioned into disjoint subset of the form $\mathbb{Z}_k(\sqrt{D})$ where $k$ ranges over the nonzero integers.  
The action of $(\mathbb{Z}_1(\sqrt{D}),\cdot)$ preserves the norm classes.

\begin{lemma}\label{lem:orbits1}
For nonzero integer $k$, if $x+y\sqrt{D} \in \mathbb{Z}_k(\sqrt{D})$, then $Orbit(x+y\sqrt{D})\subseteq \mathbb{Z}_k(\sqrt{D})$.
\end{lemma}

Next we describe the orbit structure of norm class $\mathbb{Z}_k(\sqrt{D})$. Following \cite{Nagell} call two elements of $\mathbb{Z}(\sqrt{D})$ \emph{associated} if they are in the same orbit under the action of $(\mathbb{Z}_1(\sqrt{D}),\cdot)$.
 The following theorem provides a criterion for efficiently determining whether two elements of $\mathbb{Z}_k(\sqrt{D})$ are in the same orbit under the action of the group $(\mathbb{Z}_1(\sqrt{D}),\cdot)$, \emph{i.e.}, when two elements of $\mathbb{Z}(\sqrt{D})$ are associated:
\begin{theorem}[\cite{Nagell}, p. 205]\label{thm:SameOrbit}
Let $x + y\sqrt{D}$ and $u + v\sqrt{D}$ be elements of $\mathbb{Z}(\sqrt{D})$. The following are equivalent:
\begin{enumerate}
\item{$x + y\sqrt{D}$ and $u + v\sqrt{D}$ are in the same orbit of the action of $(\mathbb{Z}_1(\sqrt{D}),\cdot)$.}
\item{$\mathcal{N}(x+y\sqrt{D})=\mathcal{N}(u+v\sqrt{D}) = m$ (say), and both $\frac{xu - Dvy}{m}$ and $\frac{xv - yu}{m}$ are integers.}
\end{enumerate}
\end{theorem}

To aid  in the analysis of the relation between a norm class and the orbits in that norm class, we define the notion of a fundamental solution in an orbit\footnote{Note the distinction between a fundamental element of a norm class, and a fundamental solution in an orbit in a norm class.}:  
\begin{definition}\label{def:FundSol}
\begin{enumerate}
\item{For element $\alpha$ of $\mathbb{Z}(\sqrt{D})$ we say that $Orbit(\alpha)$ and $Orbit(\overline{\alpha})$ are \emph{conjugate} orbits. If
$Orbit(\alpha) = Orbit(\overline{\alpha})$, then $Orbit(\alpha)$ is said to be an \emph{ambiguous} orbit.}
\item{Let $Orbit(\alpha)$ be an orbit for the action of $(\mathbb{Z}_1(\sqrt{D}),\cdot)$ on the set $\mathbb{Z}(\sqrt{D})$.
\begin{enumerate}
\item{If the orbit is not ambiguous, then among all members of the orbit  a member $u + v\sqrt{D}$  for which $v$ has minimum positive value is said to be a \emph{fundamental solution of the orbit}.}  
\item{If the orbit is ambiguous, then a member $u + v\sqrt{D}$ is a fundamental solution if $v$ has minimum positive value, and $u\ge 0$.}
\end{enumerate}}
\end{enumerate}
\end{definition}

\begin{lemma}\label{lemma:UniqueFundSol}
Each orbit has a unique fundamental solution.
\end{lemma}

The hunt for the fundamental solution of an orbit is facilitated by the following two theorems (one treating $k>0$ and the other, $k<0$) - see for example \cite{Nagell}:

\begin{theorem}\label{thm:FundSolConstraints1}
If for positive integer $k$ the element $u + v\sqrt{D}$ is the fundamental solution of an orbit that is a subset of $\mathbb{Z}_k(\sqrt{D})$, and if $x_1 + y_1\sqrt{D}$ is the fundamental element of norm class $\mathbb{Z}_1(\sqrt{D})$, then
\[
  0 \le v \le \frac{y_1\sqrt{k}}{\sqrt{2x_1+2}}
\]
and 
\[
  0 \le \vert u\vert \le \sqrt{\frac{k(x_1+1)}{2}}
\] 
\end{theorem}

\begin{theorem}\label{thm:FundSolConstraints2} 
If for negative integer $k$ the element $u + v\sqrt{D}$ is the fundamental solution of an orbit that is a subset of $\mathbb{Z}_k(\sqrt{D})$, and if $x_1 + y_1\sqrt{D}$ is the fundamental element of norm class $\mathbb{Z}_1(\sqrt{D})$, then
\[
  0 \le v \le \frac{y_1\sqrt{\vert k\vert}}{\sqrt{2x_1-2}}
\]
and 
\[
  0 \le \vert u\vert \le \sqrt{\frac{\vert k\vert(x_1-1)}{2}}
\] 
\end{theorem}
Note that the inequalities in Theorems \ref{thm:FundSolConstraints1} and \ref{thm:FundSolConstraints2} are true for $Orbit(\alpha)$ if, and only if, they are true of the conjugate orbit, $Orbit(\overline{\alpha})$. 
For any nonzero integer $k$ and any positive integer $D$ that is not a perfect square, a finite search through the candidate pairs $(u,v)$ provided by the inequalities given in Theorems \ref{thm:FundSolConstraints1} and \ref{thm:FundSolConstraints2} determines the solvability and the fundamental solutions of the Brahmagupta equation $x^2 - D y^2  = k$. In the special case when $0<\vert k\vert <\sqrt{D}$, another deterministic algorithm for solvability of Brahmagupta equation $x^2 - D y^2 = k$ is given by the continued fraction expansion of $\sqrt{D}$. More precisely,

\begin{theorem}[Classical]\label{thm:CFrelevance} If positive integers x and y solve 
\[
  \mathcal{N}(x + y\sqrt{D}) = k, 
\]
and if $k < \sqrt{D}+1$, then $\frac{x}{y}$ is necessarily a partial fraction in the continued fraction expansion of $\sqrt{D}$.
\end{theorem}
\begin{proof} See \cite{LeVeque} Theorem 9.13, \cite{Mollin}, p. 232, Theorem 5.2.5 or \cite{NZM}, Theorem 7.24.
\end{proof}

A norm class $\mathbb{Z}_k(\sqrt{D})$ may consist of a single orbit or of multiple orbits of the group action of  $(\mathbb{Z}_1(\sqrt{D}),\cdot)$. Another important consequence of Theorems \ref{thm:FundSolConstraints1} and \ref{thm:FundSolConstraints2} is: 

\begin{corollary}\label{cor:orbits}
For each nonzero integer $k$ and positive integer $D$ that is not a perfect square, the action of the group $(\mathbb{Z}_1(\sqrt{D}),\cdot)$ on the set $\mathbb{Z}_k(\sqrt{D})$ partitions $\mathbb{Z}_k(\sqrt{D})$ into finitely many orbits.
\end{corollary}

\begin{example}\label{ex29}
Consider $D = 29$. The fundamental element of $\mathbb{Z}_1(\sqrt{29})$ is $9801+1820\sqrt{29}$. Applying Theorem \ref{thm:FundSolConstraints1} for $k=140$ we find that if $x+y\sqrt{D}$ is a fundamental solution of the Brahmagupta equation $x^2 - 29 y^2 = 140$, then $x$ and $y$ satisfy the constraints $y\le 153 \mbox{ and } \vert x\vert \le 828$.

Exhaustive search through $x+y\sqrt{29}$ with $0\le \vert x\vert \le 828$ and $0\le y\le 153$ integers produces only the elements $\pm13 + 1\sqrt{29}$, $\pm16 + 2\sqrt{29}$, $\pm71 + 13 \sqrt{29}$, $\pm103 + 19\sqrt{29}$, $\pm 248 +46\sqrt{29}$ and $\pm 361 + 67\sqrt{29}$ as actual solutions of $x^2 - 29y^2 = 140$. Applying the criterion in Theorem \ref{thm:SameOrbit} to these twelve elements indicates that no two are associated, whence $12$ is the actual number of orbits constituting the norm class $\mathbb{Z}_{140}(\sqrt{29})$.
\end{example}

Next we describe a taxonomy for the orbit of an element $\alpha$ of $\mathbb{Z}_k(\sqrt{D})$ under the action of the group $(\mathbb{Z}_1(\sqrt{D}),\;\cdot)$. Let  $\beta = u+v\sqrt{D}$ be the fundamental element of $\mathbb{Z}_1(\sqrt{D})$. The group inverse of $\beta$ is $\overline{\beta} = u - v\sqrt{D}$, and the idempotent element of $\mathbb{Z}_1(\sqrt{D})$ is $-1+0\sqrt{D}$. Consider the subgroup $\langle u + v \sqrt{D}\rangle$ generated by $u+v\sqrt{D}$ 
of the group $(\mathbb{Z}_1(\sqrt{D}),\;\cdot)$. 
The cosets of the subgroup $\langle u + v \sqrt{D}\rangle$ are the subgroup $\langle u + v\sqrt{D}\rangle$ and the set $(-1)\cdot\langle u + v\sqrt{D}\rangle$. Each of these two cosets subdivide naturally into the following four sets:

 For convenience, declare $\beta_1 = \beta (= u+v\sqrt{D})$, $\beta_2 = \overline{\beta} (= u- v\sqrt{D})$. 
 Define 
\begin{itemize}
\item[]{$T_1 =  \{\beta_1^n:n\in\mathbb{N}\}.$} 
\item[]{$T_2 =  \{\beta_2^n:n\in\mathbb{N}\}.$} 
\item[]{$T_3 =   \{-\beta_1^n:n\in\mathbb{N}\}.$}
\item[]{$T_4 =   \{-\beta_2^n:n\in\mathbb{N}\}.$}
\end{itemize} 
Observe that the coset $\langle u + v \sqrt{D}\rangle$ 
is $T_1\cup T_2\cup\{1\}$ and the coset $(-1)\cdot \langle u + v \sqrt{D}\rangle$ is $T_3\cup T_4\cup\{-1\}$. Overall, 
$\mathbb{Z}_1(\sqrt{D})\setminus\{-1,\; 1\} = T_1\cup T_2\cup T_3 \cup T_4$, and the sets $T_1$, $T_2$, $T_3$ and $T_4$ are pairwise disjoint from each other.
Thus, the orbit of an element $\alpha = s_0+t_0\sqrt{D}$ under the action of the group $(\mathbb{Z}_1(\sqrt{D}),\;\cdot)$ decomposes into the following five pairwise disjoint sets:
\begin{itemize}
\item[]{$\alpha\cdot T_1 =  \{\alpha\cdot \beta_1^n:n\in\mathbb{N}\}.$}
\item[]{$\alpha\cdot T_2 =   \{\alpha\cdot \beta_2^n:n\in\mathbb{N}\}.$}
\item[]{$\alpha\cdot T_3 =   \{-\alpha\cdot \beta_1^n:n\in\mathbb{N}\}.$}
\item[]{$\alpha\cdot T_4 =   \{-\alpha\cdot\beta_2^n:n\in\mathbb{N}\}.$}
\item[]{$\{-\alpha,\; \alpha\}$}
\end{itemize} 

\begin{definition}\label{def:canonicalorbitseq} For an element $\alpha = s_0+t_0\sqrt{D}$, define the sequences 
\begin{itemize}
\item[]{$\tau_1(\alpha) = (\alpha\cdot \beta_1^n:n\in\mathbb{N})$} 
\item[]{$\tau_2(\alpha) = (\alpha\cdot \beta_2^n:n\in\mathbb{N})$} 
\item[]{$\tau_3(\alpha) = (-\alpha\cdot \beta_1^n:n\in\mathbb{N})$, and} 
\item[]{$\tau_4(\alpha) =  (- \alpha\cdot \beta_2^n:n\in\mathbb{N})$.}
\end{itemize}
Each of the sequences $\tau_1(\alpha)$, $\tau_2(\alpha)$, $\tau_3(\alpha)$ and $\tau_4(\alpha)$ is said to be a \emph{canonical orbit sequence} of $\alpha$. For canonical orbit sequence $\tau_i(\alpha)$, write the sequence as $(s_{i,n}+t_{i,n}\sqrt{D}:n\in\mathbb{N})$.
\end{definition}

\section{Benford's Law}\label{sec:BenfordIntro} 

For the formulation of Benford's Law for sequences of real numbers, we follow \cite{BergerHill}:
\begin{definition}\label{def:BenfordSeq}
A sequence 
 $(s_n:n\in\mathbb{N})$
is said to be a \emph{Benford Sequence} if: 
\begin{itemize}
\item{For each $n\in\mathbb{N}$, $s_n\neq 0$ and}
\item{ for each positive integer $m$, and for each
   \[ d_j \in
     \left\{\begin{array}{ll}
         \{1,\; 2,\; 3,\; \cdots,\;9\} & \mbox{when } j=1\\
         \{0,\;1,\; 2,\; \cdots,\; 9\}       & \mbox{when } 2\le j\le m, 
     \end{array}\right.
  \]
  it is true that
\begin{equation}\label{eq:Benford}
 \lim_{N\rightarrow\infty}\frac{\vert \{1\le n\le N:D_j(s_n) = d_j,\; 1\le j\le m \}\vert}{N} = \log_{10}(1+\frac{1}{\Sigma_{j=1}^m 10^{m-j}d_j}). 
\end{equation}}
\end{itemize}
\end{definition}

In \cite{Diaconis} and other papers, authors also use the terminology that  $(s_n:n\in\mathbb{N})$ is a \emph{strong Benford sequence}, to distinguish from the case where only the leading digit is considered, namely
\[
   \lim_{N\rightarrow\infty}\frac{\vert \{1\le n\le N:D_1(s_n) = d_1 \}\vert}{N} = \log_{10}(1+ \frac{1}{d_1}),
\]
a case which occurs in equation (1) on p. 554 of Benford's paper \cite{Benford}. The formulation in Definition \ref{def:BenfordSeq} reflects the more general situation considered by Benford in the subsection ``\emph{Frequency of Digits in the $q$-th Position}" on p. 555 of \cite{Benford}.

Two additional notions relevant to exploring strong Benford sequences are as follows: 
\begin{definition}\label{AsymptoticDensity}
A set $A\subset \mathbb{N}$ has asymptotic density if 
\[
  \lim_{n\rightarrow\infty}\frac{\vert\{m\in A:m\le n\}\vert}{n}
\]
exists. The limit value, when it exists, is denoted $d(A)$ and is called the (asymptotic) density of $A$.
\end{definition} 

For a real number $x$, $\lfloor x\rfloor$ denotes the largest integer less than or equal to $x$. The symbol $\lbrack x\rbrack$ denotes $x - \lfloor x\rfloor$. 
For a sequence $(s_n:n\in\mathbb{N})\in\;^{\mathbb{N}}\mathbb{R}$, and for $0\le t <1$ define
\[
  L_t((s_n:n\in\mathbb{N})) = \{n\in\mathbb{N}: \lbrack s_n\rbrack \le t\}.
\]

\begin{definition}\label{def:UnifDistr}
A sequence $(s_n:n\in\mathbb{N}) \in\; ^{\mathbb{N}}\mathbb{R}$ is said to be \emph{uniformly distributed modulo 1} if for each $t \in \lbrack 0,\;1)$,
$
  d(L_t(s_n:n\in\mathbb{N})) = t.
$
\end{definition}

The following theorem relates the concepts of uniform distribution modulo 1 and being a strong Benford sequence:

\begin{theorem}[Diaconis, 1977]\label{thm:Diaconis}
For a sequence $(s_n:n\in\mathbb{N})$ of real numbers the following are equivalent:
\begin{enumerate}
\item{$(\log_{10}(\vert s_n\vert) :n\in\mathbb{N})$ is uniformly distributed mod 1.}
\item{$(s_n:n\in\mathbb{N})$ is a strong Benford sequence.}
\end{enumerate}
\end{theorem}

\section{Benford's Law and canonical orbit sequences.}\label{sec:BenfordCanonicalOrbits}  

Next we establish the presence of Benford's Law in the canonical orbits of the action of $(\mathbb{Z}_1(\sqrt{D}),\cdot)$ on $\mathbb{Z}(\sqrt{D})$. The first step is the following theorem:

\begin{theorem}\label{thm:Recursion}
Let $D>0$ be an integer that is not a perfect square and let  $k$ and $\ell$ be nonzero integers. Let $u + v\sqrt{D}$ be an  
element of $\mathbb{Z}_{\ell}(\sqrt{D})\setminus\{-1,\; 1\}$. Let $s_0+ t_0\sqrt{D}$ be an element of $\mathbb{Z}_k(\sqrt{D})$. 
For each $n$ set $s_n + t_n\sqrt{D}  = (s_0 + t_0\sqrt{D})\cdot (u + v\sqrt{D})^n$. Then the sequence $((s_n,t_n):n\in\omega)$ satisfies the recurrence that for all $n\ge 0$,
\begin{enumerate}
\item{$ s_{n+2} = 2s_{n+1}u - s_n$, and }
\item{$ t_{n+2} = 2t_{n+1}u - t_n$.}
\end{enumerate}
\end{theorem}

\begin{proof}
With $(s_n,\;t_n)$ computed, we find
\begin{itemize}
\item[a)]{$(s_{n+1} + t_{n+1}\sqrt{D}) = (s_n + t_n\sqrt{D})\cdot (u + v\sqrt{D}) = (s_nu+Dt_nv) + (s_nv+t_nu)\sqrt{D}$, and} 
\item[b)]{$(s_{n+2} + t_{n+2}\sqrt{D}) = (s_{n+1}u+Dy_{n+1}v) + (s_{n+1}v+t_{n+1}u)\sqrt{D}$}
\end{itemize}
But then
\[
\begin{array}{lcl}
Dt_{n+1}v & = & Ds_nv^2 +Dt_nvu  \\
                 &     &  \mbox{ and } \\
  us_{n+1} & = & s_nu^2 +Dt_nvu  \\
                  &   &  \mbox{ and so}\\
Dt_{n+1}v & = & Ds_n v^2 + us_{n+1} - s_nu^2    \\
                  & = & us_{n+1} - s_n(u^2 - Dv^2) \\
                  & = & us_{n+1} - s_n  \\
 \end{array}
 \]

Consequently, 
\[
  s_{n+2} = 2s_{n+1}u - s_n. 
\]
Also
$s_{n+1}v = s_nuv+Dt_nv^2$ and $ut_{n+1}= s_nuv +t_nu^2$ and so
$s_{n+1}v = (ut_{n+1} - t_nu^2) + Dt_nv^2 = ut_{n+1} - t_n(u^2-Dv^2) = ut_{n+1} - t_n$.

But then it follows that $t_{n+2} = 2t_{n+1}u - t_n$.
\end{proof}

\begin{corollary}\label{cor:Recursion}
Let $D>0$ be an integer that is not a perfect square and let  $k$ be a nonzero integer. Let $u + v\sqrt{D}$ be an  
element of $\mathbb{Z}_{1}(\sqrt{D})\setminus\{-1,\; 1\}$. Let $s_0+ t_0\sqrt{D}$ be an element of $\mathbb{Z}_k(\sqrt{D})$. 
For each $n$ set $s_n + t_n\sqrt{D}  = (s_0 + t_0\sqrt{D})\cdot (u + v\sqrt{D})^n$. Then 
\begin{enumerate}
\item{The sequence $((s_n,t_n):n\in\omega)$ satisfies the recurrence that for all $n\ge 0$,
     \begin{itemize}
     \item{$ s_{n+2} = 2s_{n+1}u - s_n$, and }
     \item{$ t_{n+2} = 2t_{n+1}u - t_n$.}
     \end{itemize}}
\item{For each $n$, $s_n + t_n\sqrt{D}$ is a member of $\mathbb{Z}_{k}(\sqrt{D})$.}
\end{enumerate}
\end{corollary}

Observe that in Corollary \ref{cor:Recursion} the element $u+v\sqrt{D}$ of $\mathbb{Z}_1(\sqrt{D})\setminus\{-1,\;1\}$ is not necessarily the fundamental element of $\mathbb{Z}_1(\sqrt{D})$, but could be any element different from $1$ and from $-1$ of the subgroup $\langle x_1 + y_1\sqrt{D}\rangle$ of $\mathbb{Z}_1(\sqrt{D})$. Moreover, the element $s_0+t_0\sqrt{D}$ of $\mathbb{Z}_k(\sqrt{D})$ need not be a fundamental solution of its orbit, but could be any element of any of the orbits of $\mathbb{Z}_k(\sqrt{D})$.
Note in particular that the orbit of an element of $\mathbb{Z}_k(\sqrt{D})$ under any infinite subgroup of $(\mathbb{Z}_1(\sqrt{D}),\;\cdot)$ has the property given in Corollary \ref{cor:Recursion}.

Recall that when a sequence $(u_n:\; n= 1,\; 2,\; \cdots)$ satisfies, for constant integers $a_1$ and $a_2$ with $a_1\neq 0$, the recursion
 \begin{equation}\label{eq:generic2recursion}
  u_{n+2} = a_2u_{n+1} + a_1u_n, \;\;a_1\neq 0,
 \end{equation}
then there is an associated quadratic equation called its \emph{characteristic equation}, which is
\begin{equation}\label{eq:chareq}
  x^2 = a_2x + a_1,\;\; a_1\neq 0.
\end{equation}
The corresponding \emph{characteristic polynomial} is $p(x) = x^2 - a_2x-a_1$. 
The next relevant piece of information is the following theorem from \cite{NS}:

\begin{theorem}[Nagasaka and Shiue]\label{thm:charpoly}  If the characteristic equation of the order 2 recurrence (\ref{eq:generic2recursion}) has roots $\alpha$ and $\beta$ with $\vert \alpha\vert > \vert \beta\vert$ and $\alpha$ and $\beta$ are not of the form $\pm10^m$ for some nonnegative integer $m$, then the sequence $(u_n:n=1,\;2,\;\cdots)$ is a strong Benford sequence.
 \end{theorem}

\begin{theorem}\label{thm:orbitBenford}  
Let $D>0$ be an integer that is not a perfect square and let  $k$ be a nonzero integer. Let $u + v\sqrt{D}$ be an  
element of $\mathbb{Z}_{1}(\sqrt{D})\setminus\{-1,\; 1\}$. Let $s_0+ t_0\sqrt{D}$ be an element of $\mathbb{Z}_k(\sqrt{D})$. 
For each $n$ set $s_n + t_n\sqrt{D}  = (s_0 + t_0\sqrt{D})\cdot (u + v\sqrt{D})^n$. Then in the corresponding sequence $(s_n+t_n\sqrt{D}: n\ge 0)$, both the sequences $(s_n:n\ge 0)$ and $(t_n:n\ge 0)$ are strong Benford sequences.
\end{theorem}

\begin{proof}
By Corollary \ref{cor:Recursion}, both the sequence $(s_{i,n}:n\ge 0)$ and the sequence $(t_{i,n}:n\ge 0)$ in the canonical orbit sequence $\tau_i$ of $\alpha$ 
satisfies the second order linear recurrence given in (1) and (2) of that theorem. This second order linear recurrence has characteristic polynomial
\[
 p(x) =  x^2 - 2u\cdot x + 1 
\]
where $u$ is from an element $u+v\sqrt{D}$ with $u\neq \pm 1$ of the norm class $\mathbb{Z}_1(\sqrt{D})$. The characteristic polynomial $p(x)$ (of both recurrences in Theorem \ref{thm:Recursion}) has roots $\alpha = \frac{2u + \sqrt{4u^2 - 4}}{2} = u+\sqrt{u^2 - 1}$ and $\beta = \frac{2u - \sqrt{4u^2 - 4}}{2} = u - \sqrt{u^2 - 1}$. Note that
\begin{enumerate}
\item{$\alpha$ and $\beta$ are real numbers }
\item{$\vert u\vert \ge 2$, since $u+v\sqrt{D}$ is a member other than $1$ or $-1$ of norm class $\mathbb{Z}_1(\sqrt{D})$ and}
\item{$\alpha>\beta$.}
\end{enumerate}

If $\alpha = \pm 10^m$ for some positive integer $m$, then we have $u + \sqrt{u^2 - 1} = \pm 2^{m}\cdot 5^m$. Consider the case $\alpha = 10^m$. Then $u^2 - 1 = (2^{m}\cdot 5^m - u)^2$, producing
\[
 u^2 - 1 = u^2 - 2^{m+1}\cdot 5^m\cdot u + 2^{2m}\cdot 5^{2m},
\]
and so 
\[
  1= 2^{m+1}\cdot 5^m\cdot u - 2^{2m}\cdot 5^{2m}.
\]
But $1$ is not divisible by $5$, and so $m=0$, meaning $1 =  2u - 1$. But then $u=1$, a contradiction. A similar argument applies to the case $\alpha = - 10^m$.

Thus, if $D$ is a positive integer that is not a perfect square, then Theorem \ref{thm:charpoly} implies that both $(s_n:n\in {\mathbb N})$ and $(t_n: n\in {\mathbb N})$ are strong Benford sequences.
\end{proof}

The following example gives a hint of the relevance of the preceding remarks and results to continued fractions of $\sqrt{D}$ for $D>0$ that is not a perfect square.
\begin{example}\label{Ex:D77}
Consider $D = 77$ and $k=-13$. By Theorem \ref{thm:FundSolConstraints2} if $u+v\cdot \sqrt{77}$ is a fundamental solution of an orbit in norm class $\mathbb{Z}_{-13}(\sqrt{77})$, then $\vert u\vert <48$ and $0\le v< 6$. Exhaustive verification reveals that only $8+ 1\cdot \sqrt{77}$ and $-8 + 1\cdot \sqrt{77}$ are candidates for fundamental solutions in $\mathbb{Z}_{-13}(\sqrt{77})$. By Theorem \ref{thm:SameOrbit} these two elements are not associated, and thus the norm class $\mathbb{Z}_{13}(\sqrt{77})$ consists of exactly two orbits. Among the partial quotients in the continued fraction expansion of $\sqrt{77}$, one produces the element $272 + 31\cdot \sqrt{77}$ of $\mathbb{Z}_{-13}(\sqrt{77})$ and another produces the element $8 + 1\cdot\sqrt{77}$. An application of Theorem \ref{thm:SameOrbit} shows that the element $272 + 31\cdot \sqrt{77}$ is not in the orbit of $8+1\cdot \sqrt{77}$, but is in the orbit of $-8 +1\cdot \sqrt{77}$. Thus the partial fractions for the continued fraction expansion of $\sqrt{77}$ contains the entire orbit of the element $8+ 1\cdot \sqrt{77}$ of the norm class $\mathbb{Z}_{-13}(\sqrt{77})$, but only a suborbit of the orbit of the element $-8+1\cdot \sqrt{77}$. An application of Corollary \ref{cor:Recursion} shows that in each of the two cases the sequence of numerators as well as the sequence of denominators of the partial fractions associated with norm class $\mathbb{Z}_{-13}(\sqrt{77})$ satisfies the strong Benford law.
\end{example}

\section{The interleaving theorem for Benford Sequences}\label{sec:Interleaving}

Let $\sigma_1$, $\sigma_2$, $\cdots$, $\sigma_k$ be $k$ given sequences of objects. For $1\le j\le k$ write
\[
  \sigma_j = (s_{j,1},\;s_{j,2},\; \cdots,\;s_{j,n},\; \cdots)
\]
for the explicit format of the sequence $\sigma_j$. Define a new sequence
\[
  z =\mathcal{I}(\sigma_1,\;\sigma_2,\;\cdots,\;\sigma_k) 
\]
by
$z_{k(n-1)+i} = s_{i,n}$ for all $n$ and $i\le k$. $\mathcal{I}(\sigma_1,\cdots,\sigma_k)$ is said to be the \emph{interleavement} of the vector $(\sigma_1,\; \cdots,\;\sigma_k)$ of sequences.

\begin{theorem}[Interleaving Theorem]\label{thm:conjunction}
Let $k$ be a positive integer, and for $1\le j\le k$ let $\sigma_j = (s_{j,n}:n\in\mathbb{N})$ be a strong Benford sequence of real numbers. Then the sequence $\mathcal{I}(\sigma_1,\cdots,\sigma_k)$ is a strong Benford sequence.
\end{theorem}
\begin{proof} A direct proof can be given, but instead we use known results. Let $(z_n:n\in\mathbb{N})$ denote $\mathcal{I}(\sigma_1,\cdots,\sigma_k)$.
By Theorem 1 of \cite{Diaconis}, as each $\sigma_j$ is a strong Benford sequence, the sequence $\mu_j = (\log_{10}(s_{j,n}):n\in\mathbb{N})$ is uniformly distributed mod 1. Then by Corollary 2 of \cite{MP}, the sequence $(\log_{10}(z_n):n\in\mathbb{N})$ is uniformly distributed mod 1. Thus, by Theorem 1 of \cite{Diaconis}, $(z_n:n\in\mathbb{N})$ is a strong Benford sequence.
\end{proof}

\begin{theorem}\label{thm:Benford2} Let $D$ be a positive integer that is not a perfect square. Assume that $k$ is a nonzero integer for which the norm class $\mathbb{Z}_k(\sqrt{D})$ is nonempty, and is partitioned into $m$ orbits under the action of $\mathbb{Z}_1(\sqrt{D})$. Let 
\[
  \tau_{1,1},\tau_{1,2},\tau_{1,3},\tau_{1,4},\cdots,\tau_{m,1},\;\tau_{m,2},\;\tau_{m,3},\;\tau_{m,4}
\]
 be the corresponding canonical orbit sequences constituting the norm class $\mathbb{Z}_k(\sqrt{D})$. Let $\sigma_1,\cdots,\sigma_{4m}$ be a listing, in any order, of these canonical orbit sequences. List the elements of $\mathbb{Z}_k(\sqrt{D})$ according to the enumeration in $\mathcal{I}(\sigma_1,\cdots,\sigma_{4m})$, say as $(X_m+Y_m\sqrt{D}:m\in\mathbb{N})$. Then both the sequences $(X_m:\;m\in\mathbb{N})$ and $(Y_m:\;m\in\mathbb{N})$ are strong Benford sequences.
\end{theorem}
\begin{proof} 
For each $j$ with $1\le j\le 4m$, let $(s_{j,n}+t_{j,n}\sqrt{D}:n\in\mathbb{N})$ be the sequence $\sigma_j$. By Theorem \ref{thm:orbitBenford}, for  each $j$ both $\sigma_{j,1} = (s_{j,n}:n\in\mathbb{N})$ and $\sigma_{j,2} = (t_{j,n}:n\in\mathbb{N})$ are strong Benford sequences.
Since $(X_m:m\in\mathbb{N}) = \mathcal{I}(\sigma_{1,1},\cdots,\sigma_{4m,1})$ and $(Y_m:m\in\mathbb{N}) = \mathcal{I}(\sigma_{1,2},\cdots,\sigma_{4m,2})$, Theorem \ref{thm:conjunction} implies that both $(X_m:m\in\mathbb{N})$ and $(Y_m:m\in\mathbb{N})$ are strong Benford sequences. \end{proof} 

Evidently the argument in the proof of Theorem \ref{thm:Benford2} can be applied to any finitely long interleaving combination of strong Benford sequences. In this interleaving, one and the same sequence may be represented by several (or all) of the $\sigma_i$. It is worth noting that the converse of Theorem \ref{thm:Benford2} does not hold. This can be seen as follows: In Theorem 4 of \cite{MP} the authors produce a sequence $C = (u_n:n\in\mathbb{N})$ of real numbers which is uniformly distributed mod 1, but  for any fixed $k\ge 2$ and any positive integer $j$, the subsequence $(u_{k(n-1)+j}: n\in\mathbb{N})$ is not uniformly distributed mod 1. But then by Theorem \ref{thm:Diaconis}, the sequence $(10^{u_n}:n\in\mathbb{N})$ is a strong Benford sequence, but none of the subsequences  $(10^{u_{k(n-1)+j}}: n\in\mathbb{N})$ is a strong Benford sequence.

\end{document}